\documentclass[preprint,11pt]{elsarticle}

\usepackage[T1]{fontenc}
\usepackage{lmodern}
\usepackage{amsmath,amssymb,amsthm,mathtools,mathrsfs}
\usepackage{geometry}
\usepackage{xcolor}
\usepackage{caption}
\usepackage{tocloft}   
\usepackage{hyperref}  
\usepackage{adjustbox}
\makeatletter

\renewcommand{\tmptocnumberline}[1]{%
	\setbox0=\hbox{\appendixname}%
	\appnamewidth=\wd0
	\addtolength{\appnamewidth}{2.5pc}%
	\hb@xt@\appnamewidth{#1.\hfill}}
\makeatother

\hypersetup{colorlinks=true,linkcolor=blue!55!black,urlcolor=blue!55!black}
\allowdisplaybreaks[4]

\newtheorem{theorem}{Theorem}[section]
\newtheorem{lemma}[theorem]{Lemma}
\newtheorem{proposition}[theorem]{Proposition}
\newtheorem{corollary}[theorem]{Corollary}
\theoremstyle{remark}
\newtheorem{remark}[theorem]{Remark}

\newcommand{\cD}{\mathcal D}

\def\P{A_n(x)}
\def\Q{A_{n+1}(x)}
\def\R{A_{n+2}(x)}

\newdimen\Squaresize \Squaresize=11pt
\newdimen\Thickness \Thickness=0.7pt
\def\Square#1{\hbox{\vrule width \Thickness
		\vbox to \Squaresize{\hrule height \Thickness\vss
			\hbox to \Squaresize{\hss#1\hss}
			\vss\hrule height\Thickness}
		\unskip\vrule width \Thickness} \kern-\Thickness}

\def\Vsquare#1{\vbox{\Square{$#1$}}\kern-\Thickness}

\begin{document}
	\begin{frontmatter}
		
		\title{Eulerian triangle is totally positive of order 3}

		\author[]{Jianxi Mao\corref{cor1}}
		\ead{maojx@dlut.edu.cn}
		\cortext[cor1]{Corresponding author}
		
		\address{School of Mathematical Sciences, Dalian University of Technology, Dalian 116024, P. R. China}
		
		\begin{abstract}
		Brenti conjectured that the Eulerian triangle is totally positive.
		In this paper, we prove that the Eulerian triangle is totally positive
		of order \(3\), thereby providing a partial affirmative answer to Brenti's
		conjecture.
		In addition, we establish total positivity of order 3 for several families of generalized Eulerian triangles.
		\end{abstract}
		
	\end{frontmatter}
\section{Introduction}
Following Karlin~\cite{Kar68},
a finite or infinite matrix is called {\it totally positive} (or {\it TP} for short)
if its minors of all orders are nonnegative.
The matrix is called {\it totally positive of order $r$} (or {\it TP$_r$} for short)
if its minors of all orders at most $r$ are nonnegative.
Total positivity is a concept of considerable power
that plays an important role in various domains of mathematics, statistics and mechanics~\cite{Kar68,Pin10}.

The Eulerian numbers are classical objects in enumerative
combinatorics, and they have been extensively studied from combinatorial,
algebraic and geometric viewpoints~\cite{Pet15}.
The (ordinary) Eulerian numbers $A(n,k)$ count the number of permutations of  $[n+1]$ with exactly $k$ descents.
Numerous generalizations and refinements arise from finite Coxeter groups, Stirling permutations, and  wreath products~\cite{Bre94,HV12,Ste94}. 
A long-standing conjecture of Brenti~\cite[Conjecture 6.10]{Bre96} states that the Eulerian triangle 
$$A=[A(n,k)]_{n,k\ge 0}=
\left[\begin{array}{rrrrrr}
	1 &  &  &  &  &  \\
	1 & 1 &   &   &   &   \\
	1 & 4 & 1 &   &   &   \\
	1 & 11 & 11 & 1 &   &   \\
	1 & 26 & 66 & 26 & 1 &   \\
	\vdots &  & &  &  & \ddots \\
\end{array}\right]$$
is totally positive.
Brenti~\cite{Bre95} established a systematic connection between total
positivity and combinatorial matrices. Since then, a wide variety of
combinatorial matrices have been shown to be totally positive; see, for
example, the work of Chen, Liang, and Wang~\cite{CLW15,CLW15B}, Wang and
Yang~\cite{WY18}, Galvin and Pacurar~\cite{GP20}, Sokal and 
collaborators~\cite{DebSokal23,DebSokal25,PSZ23,Sokal22}, and
Zhu~\cite{Zhu20}.

The conjecture has long served as a motivating open
problem.
It has inspired extensive research on the total positivity of
related combinatorial triangles and polynomial matrices. In recent
years, several powerful approaches to total positivity have been
developed and refined, including Neville elimination and determinantal
criteria~\cite{GP92,GP93,Pena01}, production-matrix
methods~\cite{ChenDeb21,DebEtAl23,Sokal22}, and continued-fraction
techniques~\cite{PS21,PSZ23}. 
Recently, Zhang~\cite{Zhang25} proved the strict positivity of a family of initial minors of the Eulerian triangle.
Despite these advances, to the best of our
knowledge, only limited progress has been made toward resolving the
conjecture itself.

Throughout the paper, all row and column indices start from 0.  
Let $\{i_0,\ldots,i_s\}_<$ denote the elements listed in increasing order.
For a matrix $M=[M(n,k)]_{n,k\ge 0}$,
let $M({I,J})$ be the submatrix of $M$ obtained by selecting the rows and columns indexed by $I$ and $J$.
For index sets
\[
I=\{i_0,\ldots,i_s\}_<,\qquad J=\{j_0,\ldots,j_s\}_<,
\]
we call the minor $\det M({I,J})$ \emph{admissible} if $j_\ell\leq i_\ell$ for all $\ell.$
Our main result is the following.

\begin{theorem}\label{thm:main}
	The Eulerian triangle is totally positive of order 3.
	Moreover, every admissible minor of order at most 3 is strictly positive.
\end{theorem}

Our proof combines a reduction criterion for admissible minors
of lower triangular matrices with the real-rootedness of Eulerian
polynomials.
By the same method,
we further extend this result to a family of Eulerian-type
triangles whose row generating polynomials satisfy a differential
recurrence. As applications, we obtain total positivity of order~$3$
for the Eulerian triangle of type $B$, the $m$-Stirling Eulerian
triangles, and the $r$-colored Eulerian triangles.

\begin{remark}
Applying this method, we obtained, with the assistance of \emph{GPT-5.6 Sol} and \emph{Rethlas}, a computer-assisted argument suggesting that the Eulerian triangle is TP$_4$.
The argument reduces the verification to the computation of 126 minors of order 4, all of which pass symbolic checks. 
However, because the AI-generated argument has not been independently verified by hand, we do not regard it as a proof of TP$_4$.
\end{remark}

The paper is organized as follows.
In Section~2, we present a criterion for total positivity of finite
order for lower triangular matrices and collect the necessary
properties of real-rooted polynomials. In Section~3, we prove
Theorem~1.1 by first proving total positivity of order 2 and then
that of order 3.
Section~4 extends the method to a family of Eulerian-type triangles and presents several
applications.
The detailed verification used in Section~3 is given in Appendix A.

\section{Background}\label{sec:fekete}
In this section, we collect several preliminary results on totally
positive matrices and real-rooted polynomials. In the first subsection,
we present a criterion for a lower-triangular matrix to be totally positive
of order $r$. 
In the second subsection, we derive a family of identities
for real-rooted polynomials.

\subsection{A criterion for the total positivity of order $r$}

\begin{lemma}
	\label{lem:structural-zero}
	Let $I=\{i_0,\ldots,i_s\}_<$ and $J=\{j_0,\ldots,j_s\}_<$.  
	For a lower triangular matrix $M$,
	if
	$i_\ell<j_\ell$ for some $\ell$, then $\det M({I,J})=0$.
\end{lemma}
We shall therefore restrict attention to admissible minors, i.e., the row-column index pair
$(I,J)$ satisfying $i_\ell\ge j_\ell$ for all $\ell$.
Note that if $s=0$, this reduces to the fact that 
the entry $M(i,j)=0$  if $i<j$ for a lower triangular matrix $M$.

\begin{lemma}\label{lem:zero-cut}
	Let $M$  be a lower triangular matrix and let $I=\{i_0,\ldots,i_s\}_<$ and $J=\{j_0,\ldots,j_s\}_<$.  
	Suppose that $i_\ell\ge j_\ell$ for all $\ell$.
	If $j_{\ell'+1}>i_{\ell'}$
	for some $0\leq \ell'<s$, then
	\begin{equation}\label{eq:zero-cut-factorization}
		\det M({I,J})
		=\det M\left({\{i_0,\ldots,i_\ell'\}_<,\{j_0,\ldots,j_\ell'\}_<}\right)\cdot
		\det M\left({\{i_{\ell'+1},\ldots,i_s\}_<,\{j_{\ell'+1},\ldots,j_s\}_<}\right).
	\end{equation}
\end{lemma}

\begin{proof}
	The first $\ell'+1$ selected rows and the last $s-\ell'$ selected columns form a
	zero block.  The selected matrix is therefore block lower triangular,
	and \eqref{eq:zero-cut-factorization} follows.
\end{proof}

The following determinantal identity is given in~\cite[(1.2)]{Pin10}.

\begin{lemma}\label{delete}
	Let $M$ be a matrix.
	Let $I=\{i_0,\ldots, i_{s-1}\}_<$ and $J=\{j_0,\ldots,j_{s}\}_<$.
	Then, for any $k\in\{0,\ldots,s-1\}$ and $\ell\in\{1,\ldots,s-1\},$
	we have
	\begin{align*}
		&\det M\left({I,J\setminus\{j_\ell\}}\right)\cdot
		\det M\left({I\setminus\{i_k\},J\setminus\{j_0,j_{s}\}}\right)\\
		&\quad=
		\det M({I,J\setminus\{j_0\}})\cdot \det M({I\setminus\{i_k\},J\setminus\{j_\ell,j_{s}\}})\\
		&\qquad+
		\det M({I,J\setminus\{j_{s}\}})\cdot
		\det M({I\setminus\{i_k\},J\setminus\{j_0,j_\ell\}}).
	\end{align*}
\end{lemma}

The following result is due to Fekete;
see, for example,~\cite[Lemma~2.1]{Pin10}.

\begin{lemma}[Fekete's lemma]
	\label{lem:fekete}
	Let $M$ be an $n\times k$ matrix, where $n\geq k$.
	Suppose that every minor of order $k-1$ formed from the first
	$k-1$ columns of $M$ is strictly positive, and that every minor
	of order $k$ formed from consecutive rows of $M$ is strictly
	positive. Then every minor of order $k$ of $M$ is strictly positive.
\end{lemma}

Using Fekete's lemma, 
we present a criterion for TP$_r$.
\begin{theorem}
	\label{thm:structural-fekete}
	Let $M=[M(n,k)]_{0\le k\le n\le N}$ be a finite lower triangular matrix, and let
	$r\geq 1$ be fixed. Suppose that every admissible
	minor with consecutive $s$ rows and consecutive $s$ columns is strictly positive for $1\leq s\leq r$.
	Then every admissible minor of $M$ of order at most $r$
	is strictly positive. 
	In particular, $M$ is TP$_r$.
\end{theorem}

\begin{proof}
	
	We proceed by induction on the order $s$. 
	The case $s=1$ is trivial.
	Let $2\leq s\leq r$, and suppose that all admissible minors of order less than $s$ are strictly positive.
	We first consider minors formed from consecutive $s$ columns. 
	Consider the $(N-k+1)\times s$ matrix 
	$$
	B=M\left(\{k,k+1,\ldots, N\},\{k,k+1,\ldots,k+s-1\}\right)
	$$
	with $k+s-1\le N$.
	By hypothesis,
	every minor of $B$ formed from consecutive $s$ rows and $s$ columns is strictly positive.
	Let
	$
	k\leq r_0<r_1<\cdots<r_{s-2}\leq N.
	$
	Then the minor
	\[
	\det M\left({\{r_0,\ldots,r_{s-2}\}_<,\{k,k+1,\ldots,k+s-2\}}\right)
	\]
	is admissible. It is strictly positive by the
	induction hypothesis. Thus all minors of order $s-1$ of $B$
	formed from its first $s-1$ columns are strictly positive.
	By Lemma~\ref{lem:fekete}, every minor of order $s$
	of $B$ is strictly positive. 
	Hence, for the lower triangular matrix $M$, 
	every admissible minor of order $s$ formed from consecutive columns is strictly
	positive.
	
	We next consider an arbitrary admissible minor
	$M({I,J})$ of order $s$.
	Let
	$I=\{i_0,\ldots,i_{s-1}\}_<$
	and
	$
	J=\{j_0,\ldots,j_{s-1}\}_<.$
	Define 
	\[
	d(J)=j_{s-1}-j_0-(s-1).
	\]
	We proceed by induction on $d(J)$. If $d(J)=0$, then the columns
	are consecutive, and the assertion has already been proved.
	Suppose that $d(J)>0$. Choose a gap in $J$ and an integer $b$
	inside this gap. Let
	\[
	J'=\{j'_0,j'_1,\ldots,j'_{s}\}=J\cup\{b\},
	\]
	where $b=j'_\ell$ for some $1\leq\ell\leq s-1$ and
	$
	J=J'\setminus\{j'_\ell\}.
	$

	By Lemma~\ref{delete} with the first selected row deleted, we obtain
	\begin{align*}
		&\det M\left({I,J'\setminus\{j'_\ell\}}\right)\cdot
		\det M\left({I\setminus\{i_0\},J'\setminus\{j'_0,j'_{s}\}}\right)\\
		&\quad=
		\det M({I,J'\setminus\{j'_0\}})\cdot \det M({I\setminus\{i_0\},J'\setminus\{j'_\ell,j'_{s}\}})\\
		&\qquad+
		\det M({I,J'\setminus\{j'_{s}\}})\cdot
		\det M({I\setminus\{i_0\},J'\setminus\{j'_0,j'_\ell\}}).
	\end{align*}
	The three minors of order $s-1$ appearing as the second
	factors are admissible. Indeed,
	their column sets are, respectively,
	\[
	J'\setminus\{j'_0,j'_{s}\},
	\qquad
	J'\setminus\{j'_\ell,j'_{s}\}=J\setminus\{j_{s-1}\},
	\qquad
	J'\setminus\{j'_0,j'_\ell\}=J\setminus\{j_0\},
	\]
	while their row set is
	$
	(i_1,\ldots,i_{s-1}).
	$
	They are therefore strictly positive by the induction
	hypothesis on the order, since these minors have order $s-1$.
	
	The pair $(I,J'\setminus\{j_s'\})$ is admissible, and $d(J'\setminus\{j'_{s}\})<d(J)$. Therefore,
	\[
	\det M({I,J'\setminus\{j'_{s}\}})>0
	\]
	by the induction hypothesis on $d(J)$.
	Similarly, 
    $d(J'\setminus\{j'_{0}\})<d(J)$, and
   the minor
	$
	\det M({I,J'\setminus\{j'_0\}})
	$
	is nonnegative:
	it is strictly positive if it is
	admissible, and it vanishes otherwise by Lemma~\ref{lem:structural-zero}.
	Consequently,
	\[
   \det M\left({I,J}\right)
	=
	\det M\left({I,J'\setminus\{j'_\ell\}}\right)>0.
	\]
	This completes the proof.
\end{proof}

Although Theorem~2.5 is stated for finite matrices, it applies
to an infinite lower triangular matrix through its finite leading
principal truncations. Indeed, every finite minor is contained in
one such truncation.

\subsection{A family of identities for real-rooted polynomials}

Total positivity is closely related to the real-rootedness of polynomials.
A polynomial is called \emph{real-rooted} if it has only real zeros.
For a sequence $(g_i)_{i\ge 0}$,
its upper Toeplitz matrix is $[g_{j-i}]_{j\ge i\ge 0}.$
Note that $g_i=0$ if $i<0$.
Aissen, Edrei, Schoenberg and Whitney \cite{AESW51} showed that a polynomial with nonnegative coefficients is real-rooted 
if and only if the Toeplitz matrix of its coefficients is totally positive.

If a polynomial
$
\sum_{i=0}^{n} a_i x^i
$
with nonnegative coefficients is real-rooted, then the coefficients satisfy Newton's inequalities:
for \(1\leq i\leq n-1\),
\begin{equation}\label{NI}
	a_i^2
	\geq
	a_{i-1}a_{i+1}
	\left(1+\frac{1}{i}\right)
	\left(1+\frac{1}{n-i}\right).
\end{equation}

Let $f(x)$ be a real-rooted polynomial of degree $n$ with nonnegative coefficients and $f(0)=1$.
Then 
\begin{equation*}\label{eq:root-factorization}
	f(x)=\prod_{i=1}^{n}(1+r_i\,x),
\end{equation*}
with $r_i> 0$ for all $i$.
For $I\subseteq[n]=\{1,\ldots,n\}$, define
\begin{equation*}\label{eq:deletion-definitions}
	f_I(x)=\prod_{i\notin I}(1+r_i\,x)=\frac{f(x)}{\prod_{i\in I} (1+r_i\,x)},
	\qquad c_I=\prod_{i\in I}(1+r_i),
	\qquad G_j(x)=\sum_{|I|=j}c_I f_I(x),
\end{equation*}
and set $G_0(x)=f(x)$.
Note that $G_{j}(x)=0$ for $j>n$, and
\begin{equation}\label{eq:G_12}
G_1(x)=\sum_{i=1}^n\frac{(1+r_i)\,f(x)}{1+r_i\,x}, \qquad G_2(x)=\sum_{1\leq j<\ell\leq n}
\frac{(1+r_j)(1+r_\ell)\,f(x)}{(1+r_j\,x)(1+r_\ell\,x)}.
\end{equation}
The following recurrence is an iterated form of Fisk's identity; see
Fisk~\cite[Eq.~(4.2.2), pp.~104--106]{Fis08}.
We include a direct proof for completeness.

\begin{lemma}\label{lem:root-deletion}
	For $0\le j\le n$, we have
	\begin{equation}\label{eq:root-deletion-recurrence}
		(j+1)\,G_{j+1}(x)=(n-j)\,G_{j}(x)+(1-x)\,G_{j}'(x).
	\end{equation}
\end{lemma}

\begin{proof}
	For $j=0,$
	$$
	G_0'(x)=f'(x)=\sum_{i=1}^n\frac{r_i\,f(x)}{1+r_i\,x}.
	$$
	Then by~\eqref{eq:G_12},
	$$
	G_1(x)-(1-x)\,G_0'(x)=\left(\sum_{i=1}^n\frac{1+r_i}{1+r_i\,x}-\frac{(1-x)\,r_i}{1+r_i\,x}\right)\,f(x)=n\,f(x)=n\, G_0(x).
	$$
	For $j=n$, by definition,
	$G_n(x)=c_{[n]}$ and $G_n'(x)=0$.
	Thus, \eqref{eq:root-deletion-recurrence} holds for $j=0$ and $j=n$.

	Fix a subset $I\subseteq[n]$ with $|I|=j$, $1\le j\le n-1$. 
	Recall that
	$
	f_I(x)=\prod_{i\notin I}(1+r_i\,x).
	$
	Then 
	\[
	f_I'(x)=\sum_{i\notin I}r_i\,f_{I\cup\{i\}}(x) \qquad \textrm{and}\qquad
	f_I(x)=(1+r_i\,x)\,f_{I\cup\{i\}}(x) \quad \textrm{for }\,  i\notin I.
	\]
	It follows that 
	\begin{align*}
		(n-j)\,f_I(x)+(1-x)\,f_I'(x)&=\sum_{i\notin I}
		\left(f_I+(1-x)\,r_i\,f_{I\cup\{i\}}(x)\right)\\
		&=\sum_{i\notin I}
		\left[(1+r_i\,x)+(1-x)\,r_i\right]f_{I\cup\{i\}}(x)\\
		&=\sum_{i\notin I}(1+r_i)\,f_{I\cup\{i\}}(x).
	\end{align*}
	
	Multiplying by $c_I$ and summing over all
	$j$-element subsets $I$ gives
	\begin{align*}
		(n-j)\,G_j(x)+(1-x)\,G_j'(x)&=\sum_{|I|=j} \left[(n-j)\,c_I\,f_I(x)+(1-x)\,c_I\,f_I'(x)\right]\\
		&=
		\sum_{|I|=j}
		\sum_{i\notin I}
		c_I\,(1+r_i)\,f_{I\cup\{i\}}(x).
	\end{align*}
	For each $(j+1)$-element subset $J\subseteq[n]$, the term
	$c_J\,f_J(x)$ occurs once for every choice of $i\in J$, with
	$I=J\setminus\{i\}$. Hence it occurs exactly $j+1$ times.
	Therefore,
	\[
	(n-j)\,G_j(x)+(1-x)\,G_j'(x)
	=(j+1)\,\sum_{|J|=j+1}c_J\,f_J
	=(j+1)\,G_{j+1}.
	\]
	This completes the proof.
\end{proof}

\section{Proof of Theorem~\ref{thm:main}}
\label{sec:orders-two-three}

The  Eulerian numbers satisfy  the recurrence relation
\begin{equation}\label{eq:eulerian-recurrence}
	A({n+1,k})=(k+1)\,A({n,k})+(n-k+2)\,A({n,k-1}),
\end{equation}
with $A(0,0)=1$.
Note that $A(n,k)=0$ for $k<0$ or $k>n$.
The Eulerian polynomial is the row generating function of the Eulerian triangle,
\begin{equation*}\label{eq:Pn-definition}
	A_n(x)=\sum_{k=0}^{n}A({n,k})\,x^k.
\end{equation*}
By~\eqref{eq:eulerian-recurrence}, the Eulerian polynomials satisfy the recurrence relation
\begin{equation}\label{eq:Pn-recurrence}
	A_{n+1}(x)=\bigl(1+(n+1)\,x\bigr)\,A_n(x)+x\,(1-x)\,A_n'(x).
\end{equation}
The following classical result is due to Frobenius~\cite{Fro10}; see also
Liu and Wang~\cite{LW07}.

\begin{lemma}\label{lem:simple-negative-roots}
	For every $n\geq1$, the polynomial $A_n(x)$ has only negative
	zeros. 
\end{lemma}

\subsection{Total positivity of order 2}
\begin{proposition}\label{tp2}
	The Eulerian triangle is TP$_2$. Moreover,
	every admissible minor of order 2 is strictly positive.
\end{proposition}

\begin{proof}
	All admissible minors of order 1 are positive, i.e., $A(n,k)>0$ for $n\ge k\ge 0$.
	By Theorem~\ref{thm:structural-fekete},  it suffices to prove that every admissible
	minor of order~$2$ formed from two consecutive rows and two
	consecutive columns is strictly positive; that is,
	\[
	\det[A(n+i,k+j)]_{0\le i,j\le1}>0
	\]
	for every $0\le k\le n$.

	For $k=0$, the minor $\det[A(n+i,j)]_{0\leq i,j\leq 1}$ equals
	$$A(n,0)\,A(n+1,1)-A(n+1,0)\,A(n,1)=A(n+1,1)-A(n,1)>0,$$
	since $A(n+1,1)=2\,A(n,1)+(n+1)\,A(n,0)=2\,A(n,1)+n+1$ by the recurrence relation~\eqref{eq:eulerian-recurrence}.
	For $k=n$, the minor $\det[A(n+i,n+j)]_{0\leq i,j\leq 1}$ equals
	$$A(n,n)\,A(n+1,n+1)-A(n+1,n)\,A(n,n+1)=A(n,n)\,A(n+1,n+1)>0.$$
	
	Newton's inequalities~\eqref{NI} imply that for
	\(1\leq k\leq n-1\),
	\[
	A(n,k)^2
	\geq
	\frac{(k+1)(n-k+1)}{k(n-k)}\,
	A(n,k-1)\,A(n,k+1).
	\]
	Since
	\[
	\frac{(k+1)(n-k+1)}{k(n-k)}
	>
	\frac{n-k+2}{n-k+1},
	\]
	we obtain
	\begin{equation*}
		(n-k+1)\,A(n,k)^2
		>
		(n-k+2)\,	A(n,k-1)\,A(n,k+1).
	\end{equation*}
	Recall that
	the Eulerian recurrence~\eqref{eq:eulerian-recurrence} gives
	\[
	A(n+1,k)
	=(k+1)\,A(n,k)+(n-k+2)\,A(n,k-1).
	\]
	Therefore, $A(n,k)\, A(n+1,k+1)-A(n,k+1)\,A(n+1,k)$ equals
	\begin{align*}
		&A(n,k)\, \left[(k+2)\,A(n,k+1)+(n-k+1)\,A(n,k)\right]\\
		&\quad -A(n,k+1)\, \left[(k+1)\,A(n,k)+(n-k+2)\,A(n,k-1)\right]\\
		&=A(n,k)\, A(n,k+1)+(n-k+1)\, A(n,k)^2-(n-k+2)\, A(n,k+1)\, A(n,k-1)>0.
	\end{align*}
	This completes the proof.
\end{proof}

\subsection{Total positivity of order 3}
For a polynomial $f(x)$, we write $[x^k]f(x)$ for the coefficient of $x^k$.
For a polynomial sequence $\mathcal F=\left(f_0(x),\ldots,f_{s-1}(x)\right)$,
denote its coefficient matrix by $F=[F(i,j)]_{0\le i\le s-1,j\ge 0}$, with entries
$$
F(i,j)=[x^j]f_i(x).
$$
For an $s$-element set
$K=\{k_0,\cdots,k_{s-1}\}_<$, write
\begin{equation*}\label{eq:coefficient-minor}
	\cD_K\left(f_0,\ldots,f_{s-1}\right)
	=\det\left([x^{k_j}]f_i(x)\right)_{0\leq i,j\leq s-1}.
\end{equation*}

\begin{lemma}
Let	$\mathcal F=\left(f_0(x),\ldots,f_{s-1}(x)\right)$ be a polynomial sequence and $g(x)$ be a polynomial.
Let $\mathcal F_g=\left(f_0(x)\,g(x),\ldots,f_{s-1}(x)\,g(x)\right)$.
The corresponding coefficient matrices of $\mathcal F$ and $\mathcal F_g$ are $F$ and $F_g$, respectively.
Suppose that  $T_g$ is the upper Toeplitz matrix of the coefficients of $g$.
Then 
$$
F_g=F\cdot T_g.
$$ 
Moreover, by the Cauchy--Binet formula,
\begin{equation}\label{eq:toeplitz-cauchy-binet}
	\cD_K\left(f_0\,g,\ldots,f_{s-1}\,g\right)
	=\sum_{|L|=s}\cD_L\left(f_0,\ldots,f_{s-1}\right)\,\det T_g(L,K).
\end{equation}
\end{lemma}

Now we are ready to prove Theorem~\ref{thm:main}.
\begin{proof}[Proof of Theorem~\ref{thm:main}]
The admissible minors of order 1 are positive by the
combinatorial interpretation of the Eulerian numbers.
Proposition~\ref{tp2} proves that every admissible minor of order 2 is strictly positive.
It remains to prove the corresponding assertion for minors of order~$3$.

Apply Lemma~\ref{lem:root-deletion} with $f(x)=A_n(x)$.
Suppose that 
\begin{equation*}
	A_n(x)=\prod_{i=1}^{n}(1+r_i\,x),
\end{equation*}
with $r_i>0$ for all $i$.
Setting $j=0$ and $j=1$ in~\eqref{eq:root-deletion-recurrence} yields
\begin{align*}
	G_1(x)=n\,A_n(x)+(1-x)\,A_n'(x),\qquad
	2\,G_2(x)=(n-1)\,G_1(x)+(1-x)\,G_1'(x).
\end{align*}
Together with recurrence \eqref{eq:Pn-recurrence}, 
$$
A_{n+1}(x)=\bigl(1+(n+1)\,x\bigr)\,A_n(x)+x\,(1-x)\,A_n'(x),
$$
we obtain the following expressions for the next two
Eulerian row generating polynomials:
\begin{align}
	\Q&=(1+x)\,\P+x\,G_1(x),\label{eq:Q-three}\\
	\R&=(1+4x+x^2)\,\P+3\,x\,(1+x)\,G_1(x)+2\,x^2\,G_2(x).\label{eq:R-three}
\end{align}
Set
\begin{equation*}\label{eq:three-blocks}
	V_1(x)=x\,(\P+G_1(x)),\qquad V_2(x)=x\,\P,
	\qquad V_3(x)=x^2\,(\P+3\,G_1(x)+2\,G_2(x)).
\end{equation*}
We suppress the argument $(x)$ whenever no confusion can arise.
Then by~\eqref{eq:Q-three} and~\eqref{eq:R-three},
$$A_{n+1}=A_n+V_1 \qquad \textrm{and} \qquad A_{n+2}=A_n+3\,V_1+V_2+V_3.$$  
Elementary row operations and
linearity in the last row give that
\begin{equation}\label{eq:three-split}
	\cD_K(A_n,A_{n+1},A_{n+2})=\cD_K(A_n,V_1,V_2)+\cD_K(A_n,V_1,V_3).
\end{equation}
By Theorem~\ref{thm:structural-fekete}, 
it suffices to prove that $\cD_K(A_n,A_{n+1},A_{n+2})>0$ for $K=\{k,k+1,k+2\}$ and $n\ge k\ge 0$.

For a subset \(I\subseteq[n]\),
let 
$$f_I=\prod_{i\notin I}(1+r_i\,x)=\frac{\P}{\prod_{i\in I}(1+r_i\,x)},$$
where $f_{\varnothing}=\P$. If $I=\{i\}$, then 
define 
$$
f_i:=f_{\{i\}}=\frac{A_n(x)}{1+r_i\,x}.
$$
We use the notation
\[
c_I=\prod_{i\in I}(1+r_i),
\qquad
g_I=\prod_{i\in I}(1+r_i\,x)=\frac{A_n(x)}{f_I},
\]
where
$
c_{\varnothing}=1.
$
By~\eqref{eq:G_12},
\begin{align*}
	G_1(x)&=\sum_{i=1}^n\frac{(1+r_i)\,\P}{1+r_i\,x}=\sum_{i=1}^{n}(1+r_i)\,f_i,\\
	 G_2(x)&=\sum_{1\leq j<\ell\leq n}
	\frac{(1+r_j)(1+r_\ell)\,\P}{(1+r_j\,x)(1+r_\ell\,x)}=\sum_{|J|=2} c_J\, f_J.
\end{align*}

For the first summand $D_K(A_n,V_1,V_2)$ on the right-hand side of~\eqref{eq:three-split}, 
\[
V_1=x\,A_n+x\,G_1=x\,\P+\sum_{i=1}^{n}(1+r_i)\,x\,f_i,
\qquad V_2=x\,A_n.
\]
The term $x\,A_n$ in the middle row gives a repeated row.  
Thus 
\begin{align*}
	\cD_K(A_n,V_1,V_2)&=\cD_K\left(A_n,\,\sum_{i=1}^{n}(1+r_i)\,x\,f_i,\,x\,A_n\right)\\
	&=\sum_{i=1}^{n}\cD_K\left(A_n,\,(1+r_i)\,x\,f_i,\,x\,A_n\right)\\
	&=\sum_{i=1}^{n}\cD_K\left((1+r_i\,x)\,f_i,(1+r_i)\,x\,f_i,\,x\,(1+r_i\,x)\,f_i\right).
\end{align*}
Let $T_{f_i}$ be the Toeplitz matrix of the coefficients of $f_i$.
By~\eqref{eq:toeplitz-cauchy-binet},
$$\cD_K(A_n,V_1,V_2)=\sum_{i=1}^{n}\sum_{|L|=3}\cD_L(1+r_i\,x,(1+r_i)\,x,x\,(1+r_i\,x))\,\det T_{f_i}(L,K).$$
Thus, $\cD_L(1+r_i\,x,(1+r_i)\,x,x\,(1+r_i\,x))=0$ unless $L=\{0,1,2\}$; for $L=\{0,1,2\}$, its value is $(1+r_i)\,r_i$.  Hence
\begin{equation}\label{eq:three-first-nonnegative}
	\cD_K(A_n,V_1,V_2)
	=\sum_{i=1}^{n}r_i\,(1+r_i)\,\det T_{f_i}({\{0,1,2\},K})\geq0,
\end{equation}
since $f_i$ is real-rooted and $T_{f_i}$ is totally positive.

For the second summand $\cD_K(A_n,V_1,V_3)$ on the right-hand side of~\eqref{eq:three-split}, 
\(V_1\) and \(V_3\) can be written explicitly as
\[
V_1
=x\,(A_n+G_1)=
x\,A_n+\sum_{i=1}^{n}(1+r_i)\, x\,f_i=\sum_{|I|\leq1}
c_I\,x\,f_I,
\]
and 
\[
V_3=x^2(A_n+3\,G_1+2\,G_2)
=\sum_{|J|\leq2}
w_{|J|}\,c_J\,x^2\,f_J,
\]
where
$(w_0,w_1,w_2)=(1,3,2).$

By linearity in the second and third rows, we have
\begin{equation}\label{eq:V1V2-multilinear-expansion}
	\mathcal D_K(A_n,V_1,V_3)
	=
	\sum_{|I|\leq1,\ |J|\leq2}
	w_{|J|}\,c_I\,c_J
	\mathcal D_K\left(A_n,x\,f_I,x^2\,f_J\right).
\end{equation}
Fix a pair \((I,J)\) occurring in this sum and put
$
U=I\cup J.
$
Since \(|I|\leq1\) and \(|J|\leq2\), we have \(|U|\leq3\).
Moreover,
\[
A_n=f_U\,g_U,
\qquad
x\,f_I=f_U\,\left(x\,g_{U\setminus I}\right),
\qquad
x^2\,f_J=f_U\,\left(x^2\,g_{U\setminus J}\right).
\]
For example, if
$
I=\{i\}$,
$
J=\{j,k\}$ with
$
i,j,k$ all {distinct},
then \(U=\{i,j,k\}\) and $\bigl(A_n,x\,f_I,x^2\,f_J\bigr)$ equals
\[
f_U\,
\left((1+r_i\,x)(1+r_j\,x)(1+r_k\,x),\,
x(1+r_j\,x)(1+r_k\,x),\,
x^2(1+r_i\,x)\right).
\]
Applying the Toeplitz--Cauchy--Binet formula~\eqref{eq:toeplitz-cauchy-binet} gives
\begin{equation}\label{eq:local-cauchy-binet}
	\mathcal D_K(A_n,x\,f_I,x^2\,f_J)
	=
	\sum_{|L|=3}
	\mathcal D_L
	\bigl(
	g_U,\,
	x\,g_{U\setminus I},\,
	x^2\,g_{U\setminus J}
	\bigr)
	\det T_{f_U}(L,K).
\end{equation}

For the degrees of the three polynomials $g_U, x\,g_{U\setminus I}$ and $x^2\,g_{U\setminus J}$,
it is easy to see that 
\[
\deg g_U=|U|\leq3,\qquad \deg\bigl(x\,g_{U\setminus I}\bigr)\leq3, \qquad \deg\bigl(x^2\,g_{U\setminus J}\bigr)\leq3.
\]
Therefore, the only possible three-element column sets $L$ in~\eqref{eq:local-cauchy-binet}
are
\[
\{0,1,2\},\qquad
\{0,1,3\},\qquad
\{0,2,3\},\qquad
\{1,2,3\}.
\]
We now group together all pairs \((I,J)\) having the same union
\(U=I\cup J\). 
Rewriting~\eqref{eq:V1V2-multilinear-expansion},
we obtain that
$\mathcal D_K(A_n,V_1,V_3)$ equals
\begin{equation}\label{eq:three-local-expansion}
	\sum_{|U|\leq3}
	\;
	\sum_{\substack{
			L\in
			\{\{0,1,2\},\{0,1,3\},\{0,2,3\},\{1,2,3\}\}}}
	\Phi_U(L)\det T_{f_U}(L,K),
\end{equation}
where
\begin{equation}\label{eq:three-Phi}
	\Phi_U(L)
	=
	\sum_{\substack{
			|I|\leq1,\ |J|\leq2\\
			I\cup J=U}}
	w_{|J|}c_Ic_J
	\mathcal D_L
	\bigl(
	g_U,\,
	x\,g_{U\setminus I},\,
	x^2\,g_{U\setminus J}
	\bigr).
\end{equation}
Thus, \(\Phi_U(L)\) is simply the total contribution of all pairs
\((I,J)\) with the same union \(U\) and the same local column set
\(L\).

\begin{table}[htbp]
	\centering
	
	{\small
		\renewcommand{\arraystretch}{1.35}
		\setlength{\tabcolsep}{3.5pt}
		
		\begin{adjustbox}{max width=\linewidth}
			\begin{tabular}{|c|c|c|c|c|}
				\hline
				$U$
				&
				$\Phi_U(\{0,1,2\})$
				&
				$\Phi_U(\{0,1,3\})$
				&
				$\Phi_U(\{0,2,3\})$
				&
				$\Phi_U(\{1,2,3\})$
				\\
				\hline
				
				$\varnothing$
				&
				$1$
				&
				$0$
				&
				$0$
				&
				$0$
				\\
				\hline
				
				$\{i\}$
				&
				$c_i(3r_i+7)$
				&
				$c_i\,r_i$
				&
				$0$
				&
				$0$
				\\
				\hline
				
				$\{i,j\}$
				&
				$2c_i c_j(r_i+r_j+6)$
				&
				$3c_i c_j(r_i+r_j)$
				&
				$4c_i c_j\,r_ir_j$
				&
				$c_i c_j\,r_ir_j(r_i+r_j)$
				\\
				\hline
				
				$\{i,j,k\}$
				&
				$6c_i c_j c_k$
				&
				$2c_i c_j c_k(r_i+r_j+r_k)$
				&
				$2c_i c_j c_k(r_ir_j+r_ir_k+r_jr_k)$
				&
				$6c_i c_j c_k\,r_ir_jr_k$
				\\
				\hline
			\end{tabular}
		\end{adjustbox}
	}
	
	\caption{The local coefficients for three consecutive rows.}
	\label{tab:local-coefficients}
\end{table}

Label the elements of $U$ by $i,j,k$ and set
$c_i=1+r_i$, $c_j=1+r_j$, and $c_k=1+r_k$.  Direct expansion of
\eqref{eq:three-Phi} gives the following complete table.
The computations are given in Appendix~A.
Every entry in Table~\ref{tab:local-coefficients}  is nonnegative.  Since the
Toeplitz minors in \eqref{eq:three-local-expansion} are nonnegative,
\eqref{eq:three-split}--\eqref{eq:three-local-expansion} show that every
third-order minor of three consecutive rows is
nonnegative.

\medskip
\noindent
\textbf{Strict positivity.}
Take $K=\{k,k+1,k+2\}$ with $0\leq k\leq n$.  In
\eqref{eq:three-local-expansion}, the choice $U=\varnothing$ and
$L=\{0,1,2\}$ has coefficient $1$.  Hence, by~\eqref{eq:three-local-expansion},
  the expansion of the second summand $\cD_K(A_n,V_1,V_3)$ contains
\begin{equation*}\label{eq:three-strict-term}
	\det T_{A_n}(\{0,1,2\},\{k,k+1,k+2\}).
\end{equation*}
Since
\[
A_n(x)=\prod_{i=1}^{n}(1+r_i \,x)
=\sum_{j=0}^{n}e_j(r_1,\ldots,r_n)\,x^j,
\qquad r_i>0,
\]
the dual Jacobi--Trudi identity gives
\[
\det T_{A_n}(\{0,1,2\},\{k,k+1,k+2\})
=\det\bigl(e_{k-i+j}(r_1,\ldots,r_n)\bigr)_{i,j=1}^{3}
=s_{(3^k)}(r_1,\ldots,r_n).
\]
Since every Schur polynomial is a sum of monomials with nonnegative
integer coefficients and \(r_1,\ldots,r_n>0\), we have
\[
s_{(3^k)}(r_1,\ldots,r_n)>0.
\]
Therefore,
\[
\det T_{A_n}(\{0,1,2\},\{k,k+1,k+2\})>0.
\]
All other terms are nonnegative.
Hence, $\cD_K(A_n,A_{n+1},A_{n+2})>0$ for $K=\{k,k+1,k+2\}$ and $n\ge k\ge 0$.
This completes the proof by Theorem~\ref{thm:structural-fekete}.
\end{proof}

\section{Extensions to Eulerian-type triangles}

We present a sufficient condition for total positivity of order~$3$ for Eulerian-type triangular matrices.

\begin{theorem}
\label{thm:uniform}
	Let $B=[B(n,k)]_{n\ge k\ge0}$ be a
lower-triangular matrix, and let
\[
B_n(x)=\sum_{k=0}^n B(n,k)x^k
\]
be its $n$th row generating polynomial. 
Suppose that
\begin{equation}
	B_{n+1}(x)
	=\bigl(1+(n\,\beta+\alpha_n)\,x\bigr)\,B_n(x)
	+\beta x\,(1-x)\,B_n'(x),
	\label{eq:recurrence}
\end{equation}
with $B_0(x)=1.$
Assume that $\beta>0$ and that, for every
$n\ge0$,
\begin{equation}
	\alpha_n>0,
	\qquad
	\alpha_{n+1}-\alpha_n+\beta\geq0
	\label{eq:conditions}.
\end{equation}
Then $B$ is totally positive of order 3.
Moreover, every admissible minor of order at most 3 is strictly positive.
\end{theorem}

When $\beta=1$ and $\alpha_n=1$ for all $n\geq0$, recurrence~\eqref{eq:recurrence}
reduces to the Eulerian recurrence~\eqref{eq:Pn-recurrence}, and hence
$B_n(x)=A_n(x)$ for every $n\geq0$.
In the following, we list some applications. 
\begin{itemize}
	\item[\rm(i)] $\beta=2$ and $\alpha_n=1$ for all $n$: the number $B(n,k)$ is the Eulerian number
	of type $B$, counting the signed permutations of order $n$ in the hyperoctahedral group with exactly $k$ descents~\cite{Bre94}.
	\item[\rm(ii)] $\beta=1$ and $\alpha_n=(m-1)n+m$ for all $n$ and for a fixed $m\ge 1$:
	the number $B(n,k)$ counts $m$-Stirling permutations of order $n+1$
	 with exactly $k+1$ descents~\cite{HV12}.
	\item[\rm(iii)] $\beta=r$ and $\alpha_n=r-1$ for all $n$ and for a fixed $r\ge 2$:
	the number $B(n,k)$ counts $r$-colored permutations of order $n$ with exactly $k$ descents~\cite{Ste94}.
\end{itemize}

\begin{corollary}
	The Eulerian triangle of type $B$, the $m$-Stirling Eulerian
	triangles for $m\ge 1$, and the $r$-colored Eulerian triangles for $r\ge 2$
	are totally positive of order 3.
	Moreover, every admissible minor of order at most 3 is strictly positive. 
\end{corollary}

Liu and Wang established the following criterion for real-rootedness.
\begin{lemma}[{\cite[Proposition~3.5]{LW07}}]\label{LWRZ}
		Let $\{P_n(x)\}_{n\geq 0}$ be a sequence of polynomials with
		nonnegative coefficients such that
		$
		\deg P_n=\deg P_{n-1}+1.
		$
		Suppose that
		\[
		P_n(x)
		=(a_nx+b_n)P_{n-1}(x)
		+x(c_nx+d_n)P_{n-1}'(x),
		\]
		where $a_n,b_n\in \mathbb{R}$, $c_n\leq 0,$ and $d_n\geq 0.$
		Then every polynomial $P_n(x)$ is real-rooted.
\end{lemma}

Now we are ready to prove Theorem~\ref{thm:uniform}.
The proof is similar to that of Theorem~\ref{thm:main},
and we will omit some details.

\begin{proof}[Proof of Theorem~\ref{thm:uniform}]
All admissible minors of order 1 are positive.
We first prove that every admissible minor of order 2 is strictly positive.
By Theorem~\ref{thm:structural-fekete}, 
it suffices to prove that
$B(n,k)\, B(n+1,k+1)-B(n,k+1)\,B(n+1,k)>0$ for all $n\ge k\ge 0$.

It is straightforward to verify that the two boundary cases
$k=0$ and $k=n$ hold.
For $1\leq k\leq n-1$, recurrence~\eqref{eq:recurrence}
gives that $B(n,k)\, B(n+1,k+1)-B(n,k+1)\,B(n+1,k)$ equals
\begin{align*}
	\beta B(n,k)\, B(n,k+1)
	+\bigl(\alpha_n+\beta(n-k)\bigr) B(n,k)^2
	-\bigl(\alpha_n+\beta(n-k+1)\bigr) B(n,k-1) B(n,k+1).
\end{align*}
By Lemma~\ref{LWRZ},
we obtain that $B_n(x)$ is real-rooted. 
Then
\[
B_n(x)=\prod_{i=1}^{n}(1+r_i\, x),
\]
with $r_i>0$ for all $i$.
Hence, by Newton's inequalities~\eqref{NI},
\[
B(n,k)^2
\geq
\frac{(k+1)(n-k+1)}{k(n-k)}
B(n,k-1)B(n,k+1).
\]
Furthermore,
\begin{align*}
	&\bigl(\alpha_n+\beta(n-k)\bigr)
	(k+1)(n-k+1)
	-\bigl(\alpha_n+\beta(n-k+1)\bigr)k(n-k)\\
	&=
	\alpha_n(n+1)
	+\beta(n-k)(n-k+1)>0.
\end{align*}
It follows that
\[
\bigl(\alpha_n+\beta(n-k)\bigr)B(n,k)^2
>
\bigl(\alpha_n+\beta(n-k+1)\bigr)B(n,k-1)B(n,k+1).
\]
Hence $B(n,k)\, B(n+1,k+1)-B(n,k+1)\,B(n+1,k)>0$ and $[B(n,k)]_{n\ge k\ge 0}$ is TP$_2$.

It remains to prove that
every admissible minor of order 3 is strictly positive.
Apply Lemma~\ref{lem:root-deletion} with $f(x)=B_n(x)$.
	For \(I\subseteq[n]\), set
	\[f_i=\frac{B_n(x)}{1+r_i\,x},\qquad
	f_I(x)=\frac{f(x)}{\prod_{i\in I}(1+r_i\, x)},\qquad
	c_I=\prod_{i\in I}(1+r_i),\qquad
	G_j=\sum_{\substack{I\subseteq[n]\\ |I|=j}}c_I \,f_I,
	\]
	and write \(c_i=1+r_i\).  
	Taking $j=0$ and $j=1$ in~\eqref{eq:root-deletion-recurrence} gives
	\begin{align*}
		G_1(x)=n\,B_n(x)+(1-x)\,B_n'(x),\qquad
		2\,G_2(x)=(n-1)\,G_1(x)+(1-x)\,G_1'(x).
	\end{align*}
	Together with recurrence \eqref{eq:recurrence}, 
	we obtain the following identities:
	\begin{align*}
		B_{n+1}
		&=B_n+x\,\bigl(\alpha_n \,B_n+\beta\, G_1\bigr),                                   \\
		B_{n+2}
		&=B_n+x\,\Bigl([\alpha_{n+1}+\beta+\alpha_{n}\,(1+\beta)]\,B_n+\beta\,(\beta+2)\,G_1\Bigr)             \notag\\
		&\quad+x^2\,\Bigl(\alpha_{n}\,\alpha_{n+1}\,B_n+\beta\,(\alpha_{n}+\alpha_{n+1}+\beta)\,G_1+2\,\beta^2\,G_2\Bigr).        
	\end{align*}

	Define
	\[
	V=\alpha_{n}\,B_n+\beta\, G_1,\qquad
	Q=\alpha_{n}\,\alpha_{n+1}\,B_n+\beta\,(\alpha_{n}+\alpha_{n+1}+\beta)\,G_1+2\,\beta^2\,G_2.
	\]
	Then 
	\[
	B_{n+1}=B_n+x\,V,\qquad
	B_{n+2}=B_n+(\beta+2)\,x\,V+(\alpha_{n+1}-\alpha_{n}+\beta) \,x\,B_n+x^2\,Q.
	\]
	Consequently, for any three-column set \(K\), elementary row operations
	and multilinearity give
	\begin{equation*}
		D_K(B_n,B_{n+1},B_{n+2})
		=\beta\,(\alpha_{n+1}-\alpha_{n}+\beta)\,D_K(B_n,x\,G_1,xB_n)+D_K(B_n,x\,V,x^2Q).            
	\end{equation*}
	The first summand on the right is nonnegative, since $\beta>0, \alpha_{n+1}-\alpha_{n}+\beta\ge 0$ and
	\begin{equation*}
		D_K(B_n,x\,G_1,x\,B_n)
		=\sum_{i=1}^{n}c_i r_i
		\det T_{f_i}(\{0,1,2\},K)\geq0.                 
	\end{equation*}
	
	For the second summand, 
	\[
	V=\alpha_{n}\,B_n+\beta\, G_1
	=\sum_{|I|\leq1}\mu_{|I|}c_I f_I,
	\qquad (\mu_0,\mu_1)=(\alpha_{n},\beta),
	\]
	and $Q=\alpha_{n}\,\alpha_{n+1}\,B_n+\beta\,(\alpha_{n}+\alpha_{n+1}+\beta)\,G_1+2\,\beta^2\,G_2$ equals
	\[
	\sum_{|J|\leq2}\nu_{|J|}\,c_J\, f_J,
	\qquad
	(\nu_0,\nu_1,\nu_2)
	=\bigl(\alpha_{n}\,\alpha_{n+1},\,\beta\,(\alpha_{n}+\alpha_{n+1}+\beta),\,2\,\beta^2\bigr).
	\]
	For \(U\subseteq[n]\), put
	\(g_U(x)=\prod_{i\in U}(1+r_i x)\).  For
	\(L\in\{\{0,1,2\},\{0,1,3\},\{0,2,3\},\{1,2,3\}\}\), define 
	\[
	\Phi_U(L)=
	\sum_{\substack{|I|\leq1,\ |J|\leq2\\ I\cup J=U}}
	\mu_{|I|}\nu_{|J|}\,c_I\,c_J\,
	D_L\!\left(g_U,xg_{U\setminus I},x^2g_{U\setminus J}\right).
	\]
	Multilinearity followed by the Cauchy--Binet formula yields
	\begin{equation*}
		D_K(B_n,x\,V,x^2\,Q)
		=\sum_{|U|\leq3}\ \sum_{L\in\{\{0,1,2\},\{0,1,3\},\{0,2,3\},\{1,2,3\}\}}
		\Phi_U(L)\det T_{f_U}(L,K).                            
	\end{equation*}
	We calculate $\Phi_U(L)$ in the same way as in the proof of Theorem~\ref{thm:main}.
    The
	following is the table and we omit the proof.
	\[
	\renewcommand{\arraystretch}{1.45}
	\begin{array}{c|c|l}
		U&L&\Phi_U(L)\\ \hline
		\varnothing&\{0,1,2\}&\alpha_{n}^2\alpha_{n+1}\\ \hline
		\{i\}&\{0,1,2\}&
		\beta c_i\!\left[\alpha_{n}^2+2\alpha_{n}\alpha_{n+1}+\alpha_{n}\beta(r_i+2)
		+\beta(\alpha_{n+1}+\beta)(r_i+1)\right]\\
		&\{0,1,3\}&\alpha_{n}\alpha_{n+1}\beta r_i c_i\\ \hline
		\{i,j\}&\{0,1,2\}&
		2\beta^2c_ic_j\!\left[2\alpha_{n}+\alpha_{n+1}+\beta(r_i+r_j+3)\right]\\
		&\{0,1,3\}&\beta^2c_ic_j(r_i+r_j)(\alpha_{n}+\alpha_{n+1}+\beta)\\
		&\{0,2,3\}&2\beta^2r_ir_jc_ic_j(\alpha_{n+1}+\beta)\\
		&\{1,2,3\}&\beta^2r_ir_jc_ic_j(r_i+r_j)(\alpha_{n+1}-\alpha_{n}+\beta)\\ \hline
		\{i,j,k\}&\{0,1,2\}&6\beta^3c_ic_jc_k\\
		&\{0,1,3\}&2\beta^3c_ic_jc_k(r_i+r_j+r_k)\\
		&\{0,2,3\}&2\beta^3c_ic_jc_k(r_ir_j+r_ir_k+r_jr_k)\\
		&\{1,2,3\}&6\beta^3c_ic_jc_kr_ir_jr_k
	\end{array}                                                 
	\]
	Every entry not displayed in this table is zero.
	It follows from condition~\eqref{eq:conditions} that every displayed coefficient in this table is nonnegative.
    The rest of the proof is to prove the strict positivity of $D_K(B_n,B_{n+1},B_{n+2})$,
    which is similar to that of Theorem~\ref{thm:main}.
    We leave it to the reader.
\end{proof}

\appendix
\section{Verification of \(\Phi_U(L)\) in the proof of Theorem~\ref{thm:main}}
\label{app:three-local-verification}

We give a direct verification of the values of $\Phi_U(L)$
listed in Table~\ref{tab:local-coefficients}. Recall that $(w_0,w_1,w_2)=(1,3,2)$ and
\[
\Phi_U(L)
=
\sum_{\substack{|I|\le 1,\ |J|\le 2\\ I\cup J=U}}
w_{|J|}c_Ic_J
D_L\bigl(g_U,xg_{U\setminus I},x^2g_{U\setminus J}\bigr).
\]

\subsection{The case \(U=\varnothing\)}

The only contributing pair is
\[
I=J=\varnothing.
\]
In this case,
\[
g_U=1,\qquad
x\,g_{U\setminus I}=x,\qquad
x^2\,g_{U\setminus J}=x^2.
\]
Hence,
\begin{align*}
	\mathcal D_L
	\bigl(
	g_U,\,
	x\,g_{U\setminus I},\,
	x^2\,g_{U\setminus J}
	\bigr)=\begin{cases}
		&1,\qquad \textrm{if }\, L=\{0,1,2\};\\
		&0,\qquad \textrm{otherwise;}
	\end{cases}
\end{align*}
and
$w_{|J|}c_Ic_J=w_0 c_\varnothing c_\varnothing=1$.
Then 
$\Phi_U(L)=1$ if $L=\{0,1,2\}$; otherwise, $\Phi_U(L)=0$.

\subsection{The case \(U=\{i\}\)}
The contributing pairs $(I,J)$ are
\[
(\varnothing,\{i\}),\qquad
(\{i\},\varnothing),\qquad
(\{i\},\{i\}).
\]
A direct computation gives
\[
\begin{array}{c|c|c|cccc}
	I&J&w_{|J|}c_Ic_J
	&\mathcal D_{\{0,1,2\}} &\mathcal D_{\{0,1,3\}}
	&\mathcal D_{\{0,2,3\}} &\mathcal D_{\{1,2,3\}}\\ \hline
	\varnothing&\{i\}
	&3c_i&1&0&0&0\\
	\{i\}&\varnothing
	&c_i&1&r_i&0&0\\
	\{i\}&\{i\}
	&3c_i^2&1&0&0&0
\end{array}.
\]
Summing the three contributing terms gives
\begin{align*}
	\Phi_{\{i\}}(\{0,1,2\})&=c_i (4+3c_i)=c_i (3r_i+7),\\
	\Phi_{\{i\}}(\{0,1,3\})&=c_ir_i,\\
	\Phi_{\{i\}}(\{0,2,3\})&=\Phi_{\{i\}}(\{1,2,3\})=0.
\end{align*}

These values follow from the following three cases.
\begin{itemize}
	\item For \((I,J)=(\varnothing,\{i\})\), the triple $\left(g_U,\,
	x\,g_{U\setminus I},\,
	x^2\,g_{U\setminus J}\right)$ equals 
	$
	\left(1+r_i\,x, x\,(1+r_i\,x),x^2\right),
	$
	and the coefficient matrix is 
	$$\begin{pmatrix}
		1&r_i&0&0\\
		0&1&r_i&0\\
		0&0&1&0
	\end{pmatrix}.$$
	\item For \((I,J)=(\{i\},\varnothing)\), the triple  $\left(g_U,\,
	x\,g_{U\setminus I},\,
	x^2\,g_{U\setminus J}\right)$ equals 
	$
	1+r_i\,x, x,x^2(1+r_i\,x),
	$
	and the coefficient matrix is 
	$$\begin{pmatrix}
		1&r_i&0&0\\
		0&1&0&0\\
		0&0&1&r_i
	\end{pmatrix}.$$
	\item For \((I,J)=(\{i\},\{i\})\), the triple  $\left(g_U,\,
	x\,g_{U\setminus I},\,
	x^2\,g_{U\setminus J}\right)$ equals 
	$
	1+r_i\,x, x, x^2,
	$
	and the coefficient matrix is 
	$$\begin{pmatrix}
		1&r_i&0&0\\
		0&1&0&0\\
		0&0&1&0
	\end{pmatrix}.$$
\end{itemize}

\subsection{The case \(U=\{i,j\}\)}
The contributing pairs $(I,J)$ are
\[
\begin{aligned}
	(\varnothing,\{i,j\}),\quad
	(\{i\},\{j\}),\quad
	(\{j\},\{i\}),\quad
	(\{i\},\{i,j\}),\quad
	(\{j\},\{i,j\}).
\end{aligned}
\]
A direct computation gives
\[
\begin{array}{c|c|c|cccc}
	I&J&w_{|J|}c_Ic_J
	&\mathcal D_{\{0,1,2\}} &\mathcal D_{\{0,1,3\}}
	&\mathcal D_{\{0,2,3\}} &\mathcal D_{\{1,2,3\}}\\ \hline
	\varnothing&\{i,j\}
	&2c_ic_j&1&0&-r_ir_j&-(r_i+r_j)r_ir_j\\
	\{i\}&\{j\}
	&3c_ic_j&1&r_i&r_ir_j&r_ir_j^2\\
	\{j\}&\{i\}
	&3c_ic_j&1&r_j&r_ir_j&r^2_ir_j\\
	\{i\}&\{i,j\}
	&2c^2_ic_j&1&0&0&0\\
	\{j\}&\{i,j\}
	&2c_ic^2_j&1&0&0&0
\end{array}.
\]
Summing the five contributing terms gives
\begin{align*}
	\Phi_{\{i,j\}}(\{0,1,2\})&=c_ic_j (8+2c_i+2c_j)=2c_ic_j(r_i+r_j+6),\\
	\Phi_{\{i,j\}}(\{0,1,3\})&=3c_ic_j(r_i+r_j),\\
	\Phi_{\{i,j\}}(\{0,2,3\})&=4c_ic_jr_ir_j,\\
	\Phi_{\{i,j\}}(\{1,2,3\})&=c_ic_jr_ir_j(r_i+r_j).
\end{align*}

These values follow from the following five cases.
\begin{itemize}
	\item For \((I,J)=(\varnothing,\{i,j\})\), the triple $\left(g_U,\,
	x\,g_{U\setminus I},\,
	x^2\,g_{U\setminus J}\right)$ equals 
	$$
	\left((1+r_i\,x)(1+r_j\,x), x\,(1+r_i\,x)(1+r_j\,x),x^2\right),
	$$
	and the coefficient matrix is 
	$$\begin{pmatrix}
		1&r_i+r_j&r_ir_j&0\\
		0&1&r_i+r_j&r_ir_j\\
		0&0&1&0
	\end{pmatrix}.$$
	\item For \((I,J)=(\{i\},\{j\})\), the triple $\left(g_U,\,
	x\,g_{U\setminus I},\,
	x^2\,g_{U\setminus J}\right)$ equals 
	$$
	\left((1+r_i\,x)(1+r_j\,x), x\,(1+r_j\,x),x^2\,(1+r_i\,x)\right),
	$$
	and the coefficient matrix is 
	$$\begin{pmatrix}
		1&r_i+r_j&r_ir_j&0\\
		0&1&r_j&0\\
		0&0&1&r_i
	\end{pmatrix}.$$
	\item For \((I,J)=(\{j\},\{i\})\), the triple $\left(g_U,\,
	x\,g_{U\setminus I},\,
	x^2\,g_{U\setminus J}\right)$ equals 
	$$
	\left((1+r_i\,x)(1+r_j\,x), x\,(1+r_i\,x),x^2\,(1+r_j\,x)\right),
	$$
	and the coefficient matrix is 
	$$\begin{pmatrix}
		1&r_i+r_j&r_ir_j&0\\
		0&1&r_i&0\\
		0&0&1&r_j
	\end{pmatrix}.$$
	\item For \((I,J)=(\{i\},\{i,j\})\), the triple $\left(g_U,\,
	x\,g_{U\setminus I},\,
	x^2\,g_{U\setminus J}\right)$ equals 
	$$
	\left((1+r_i\,x)(1+r_j\,x), x\,(1+r_j\,x),x^2\right),
	$$
	and the coefficient matrix is 
	$$\begin{pmatrix}
		1&r_i+r_j&r_ir_j&0\\
		0&1&r_j&0\\
		0&0&1&0
	\end{pmatrix}.$$
	\item For \((I,J)=(\{j\},\{i,j\})\),  the triple $\left(g_U,\,
	x\,g_{U\setminus I},\,
	x^2\,g_{U\setminus J}\right)$ equals 
	$$
	\left((1+r_i\,x)(1+r_j\,x), x\,(1+r_i\,x),x^2\right),
	$$
	and the coefficient matrix is 
	$$\begin{pmatrix}
		1&r_i+r_j&r_ir_j&0\\
		0&1&r_i&0\\
		0&0&1&0
	\end{pmatrix}.$$
\end{itemize}

\subsection{The case \(U=\{i,j,k\}\)}

Since \(|I|\leq1\), \(|J|\leq2\), and \(I\cup J=U\), the only
contributing pairs are
\[
(\{i\},\{j,k\}),\qquad
(\{j\},\{i,k\}),\qquad
(\{k\},\{i,j\}).
\]
A direct computation gives
\[
\begin{array}{c|c|c|cccc}
	I&J&w_{|J|}c_Ic_J
	&\mathcal D_{\{0,1,2\}} &\mathcal D_{\{0,1,3\}}
	&\mathcal D_{\{0,2,3\}} &\mathcal D_{\{1,2,3\}}\\ \hline
	\{i\}&\{j,k\}
	&2c_ic_jc_k&1&r_i&r_ir_j+r_ir_k-r_jr_k& y_1\\
	\{j\}&\{i,k\}
	&2c_ic_jc_k&1&r_j&r_ir_j+r_jr_k-r_ir_k&y_2\\
	\{k\}&\{i,j\}
	&2c_ic_jc_k&1&r_k&r_ir_k+r_jr_k-r_ir_j&y_3
\end{array},
\]
where 
\begin{align*}
	&y_1=r_ir_jr_k+(r_i-r_j)r_k^2+(r_i-r_k)r_j^2,\\
	&y_2=r_ir_jr_k+(r_j-r_i)r_k^2+(r_j-r_k)r_i^2,\\
	&y_3=r_ir_jr_k+(r_k-r_i)r_j^2+(r_k-r_j)r_i^2.
\end{align*}
Summing the three contributing terms gives
\begin{align*}
	\Phi_{\{i,j,k\}}(\{0,1,2\})&=6c_ic_jc_k,\\
	\Phi_{\{i,j,k\}}(\{0,1,3\})&=2c_ic_jc_k(r_i+r_j+r_k),\\
	\Phi_{\{i,j,k\}}(\{0,2,3\})&=2c_ic_jc_k(r_ir_j+r_ir_k+r_jr_k),\\
	\Phi_{\{i,j,k\}}(\{1,2,3\})&=6c_ic_jc_kr_ir_jr_k.
\end{align*}

These values follow from the following three cases.
\begin{itemize}
	\item For \((I,J)=(\{i\},\{j,k\})\), the triple $\left(g_U,\,
	x\,g_{U\setminus I},\,
	x^2\,g_{U\setminus J}\right)$ equals 
	$$
	\left((1+r_i\,x)(1+r_j\,x)(1+r_k\,x), x\,(1+r_j\,x)(1+r_k\,x),x^2(1+r_i\,x)\right),
	$$
	and the coefficient matrix is 
	$$\begin{pmatrix}
		1&r_i+r_j+r_k&r_ir_j+r_ir_k+r_jr_k&r_ir_jr_k\\
		0&1&r_j+r_k&r_jr_k\\
		0&0&1&r_i
	\end{pmatrix}.$$
	\item For \((I,J)=(\{j\},\{i,k\})\), the triple $\left(g_U,\,
	x\,g_{U\setminus I},\,
	x^2\,g_{U\setminus J}\right)$ equals 
	$$
	\left((1+r_i\,x)(1+r_j\,x)(1+r_k\,x), x\,(1+r_i\,x)(1+r_k\,x),x^2(1+r_j\,x)\right),
	$$
	and the coefficient matrix is 
	$$\begin{pmatrix}
		1&r_i+r_j+r_k&r_ir_j+r_ir_k+r_jr_k&r_ir_jr_k\\
		0&1&r_i+r_k&r_ir_k\\
		0&0&1&r_j
	\end{pmatrix}.$$
	\item For \((I,J)=(\{k\},\{i,j\})\), the triple $\left(g_U,\,
	x\,g_{U\setminus I},\,
	x^2\,g_{U\setminus J}\right)$ equals 
	$$
	\left((1+r_i\,x)(1+r_j\,x)(1+r_k\,x), x\,(1+r_i\,x)(1+r_j\,x),x^2(1+r_k\,x)\right),
	$$
	and the coefficient matrix is 
	$$\begin{pmatrix}
		1&r_i+r_j+r_k&r_ir_j+r_ir_k+r_jr_k&r_ir_jr_k\\
		0&1&r_i+r_j&r_ir_j\\
		0&0&1&r_k
	\end{pmatrix}.$$
\end{itemize}

\end{document}